\documentclass[14pt]{amsart}

\usepackage{graphicx}
\usepackage{amssymb}
\usepackage{amsthm}
\usepackage{amsmath}
\usepackage{geometry}
\usepackage{overpic}

\usepackage[pdftex,%
            colorlinks,linkcolor=red,citecolor=blue,urlcolor=blue,%
            pdfauthor={Mikhail Skopenkov},%
            pdftitle={Draft},%
            pdfkeywords={circle, web, cyclide},
            pdfstartview=FitH,%
           ]{hyperref}

\newtheorem{theorem}{Theorem}[section]
\newtheorem{lemma}[theorem]{Lemma}

\newtheorem{claim}[theorem]{Claim}

\theoremstyle{definition}
\newtheorem{definition}[theorem]{Definition}
\newtheorem{example}[theorem]{Example}

\theoremstyle{remark}
\newtheorem*{remark*}{Remark}

\newcommand{\mscomm}[1]{}

\begin{document}

\Large



\title{A surface containing a line and a circle through each point\\ is a quadric}

\author{Fedor Nilov}
\address{Laboratory of Geometrical Methods in Mathematical Physics, Moscow State University}
\email{nilovfk AT mail DOT ru}

\author{Mikhail Skopenkov}
\address{King Abdullah University of Science and Technology}
\address{Institute for Information Transmission Problems of the Russian Academy of Sciences}
\email{skopenkov AT rambler DOT ru}

\thanks{The authors were supported in part by President of the Russian Federation grant MK-3965.2012.1. The second author was supported in part by ``Dynasty'' foundation and Simons--IUM fellowship.}

\begin{abstract} We prove that a surface in 3-dimensional Euclidean space
containing a line and a circle through each point is a quadric. We also give some particular results on the classification of surfaces containing several circles through each point.

\smallskip


\end{abstract}

\maketitle





\Large

\section{Introduction}\label{sec:introduction}

Surfaces generated by simplest curves (lines and circles) are popular subject in pure mathematics
and have applications to design and architecture
\cite{Pottmann-2007, Bo-etal}.
If a surface contains two such curves through each point
then we get a mesh on the surface.
Famous examples of such meshes 
are 
V.~G.~Shukhov's hyperboloid structures. 
A natural question is which other surfaces can be constructed from straight and circular beams.

It is well-known that a surface containing two lines through each point (\emph{doubly ruled surface}) must be a quadric.
In this paper we show that 
a smooth surface containing both a line and a circle through each point still must be a quadric; see Figure~\ref{fig1} to the left. By a \emph{smooth surface} we mean the image of an injective $C^\infty$ map $\mathbb{R}^2\to\mathbb{R}^3$ with nowhere vanishing differential.


\begin{theorem}\label{ruled} If through each point of a smooth surface in $\mathbb{R}^3$ one can draw both a straight line segment and a circular arc transversal to each other and fully contained in the surface (and continuously depending on the point)
then the surface is a piece of either a one-sheeted hyperboloid, or a quadratic cone, or an elliptic cylinder, or a plane.
\end{theorem}

\begin{figure}[hb]
\includegraphics[width=0.15\textwidth]{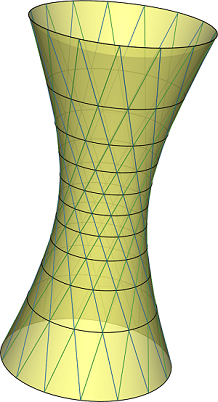}
\qquad
\includegraphics[width=0.33\textwidth]{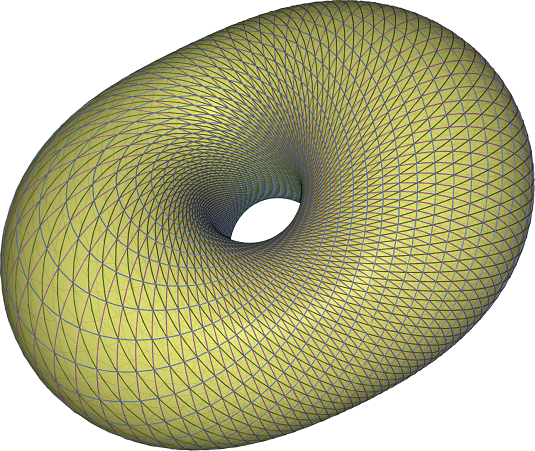}
\caption{
Left: A one-sheeted hyperboloid contains both a line and a circle through each point. To find all surfaces with this property (Theorem~\ref{ruled}), we prove that the planes of the generating circles are parallel (Lemma~\ref{parallel}) and intersect the surface only at the points of the circles (Lemma~\ref{circleonly}).
Right: A cyclide contains several circles through each point. Finding all surfaces with this property is a challenging open problem. So far we prove that a surface with $4$ circles through each point is a cyclide (Theorem~\ref{3families}).
}
\label{fig1}
\end{figure}

In what follows a line (circle) continuously depending on a real parameter is called a \emph{family} of lines (circles).

Although the proof is rather elementary, Theorem~\ref{ruled} is more tricky than the classical description of doubly ruled surfaces. Let us illustrate the difference. First, the classical result does not really require $2$ lines through \emph{each} point:
a surface covered by $1$ family of lines and containing just $3$ more lines intersecting them all must already be a quadric.
Second, the classical result remains true
in \emph{complex} 3-space.
However, similar generalizations of Theorem~\ref{ruled} are far from being true; see Examples~\ref{ex1}--\ref{ex5} below.

The next natural problem, which seems to be still open (and is going to be studied in detail in a subsequent publication), is to describe all surfaces containing several circles through each point.

An example of such surface is a \emph{cyclide}, i.e.,
the surface given by the equation of the form
\begin{equation}\label{D}
a(x^2+y^2+z^2)^2+(x^2+y^2+z^2)(bx+cy+dz)+Q(x,y,z)=0,
\end{equation}
where $a,b,c,d$ are constants and $Q(x,y,z)$ is a polynomial of degree at most $2$; see Figure~\ref{fig1} to the right. Such a surface is also called a \emph{Darboux cyclide}, not to be confused with a \emph{Dupin cyclide} being a particular case.
This class of surfaces appears in different branches of mathematics \cite{Maxwell, Moon-Spencer, Krasauskas-Zube}.
An introduction to cyclides and circles on them can be found in the work of Pottmann et~al.~\cite{Pottmann}. Any cyclide (besides some degenerate cases) 
contains at least $2$ circles through each point
~\cite{Coolidge, Pottmann}. 
Conversely, 
a surface
containing $2$ cospherical
or $2$ orthogonal circles through each point must be a cyclide; see Theorem~\ref{cospherical} below (taken from \cite{Coolidge}) and \cite[Theorems~1 and~2]{Ivey}. However, this is not true without the assumption of either cosphericity or orthogonality; see Example~\ref{ex4} below.

A torus is an example of a cyclide with $4$ circles through each point: a meridian, a parallel, and $2$ Villarceau circles. 
There are cyclides with $6$ circles through each point~\cite{Blum, Pottmann}. We give a short proof of the Takeuchi theorem \cite{Takeuchi} stating that a surface with $7$ circles through each point must be a sphere (Theorem~\ref{takeuchi} below), and we generalize an old Darboux result that a surface with sufficiently many circles through each point is a cyclide (Theorem~\ref{3families} below).

Further generalizations concern \emph{conic bundles}, in particular, surfaces containing a conic through each point. Surfaces containing both a line and a conic through each point were classified by Brauner \cite{Brauner}. Such a surface has degree at most $4$.
Notice that it is much more difficult to deduce Theorem~\ref{ruled} from this classification than to prove Theorem~\ref{ruled} itself.
Surfaces containing several conics through each point were classified by Schicho \cite{Schicho-01}. Such surfaces have degree at most $8$ and admit a biquadratic rational parametrization.
\emph{Webs} of circles and conics are discussed in \cite{Pottmann, Timorin-07, Shelekhov-07}.

The paper is organized as follows. In Section~\ref{sec:proofs} we prove Theorem~\ref{ruled}.
In Section~\ref{sec:variations} we state and prove some related results mentioned above.

\section{Proofs}\label{sec:proofs}

Throughout this section we work over the field of complex numbers except otherwise is explicitly indicated.
Denote by $\mathbb{P}^3$ the $3$-dimensional complex projective space with homogeneous coordinates $x:y:z:w$. The \emph{infinitely distant plane} is the plane $w=0$. The \emph{absolute conic} is given by the equations $x^2+y^2+z^2=0$, $w=0$. A 
(\emph{nondegenerate}) \emph{complex circle} is an irreducible conic in $\mathbb{P}^3$ having two 
distinct common points with the absolute conic.
Clearly, a circle in $\mathbb{R}^3$ is a subset of a complex circle.

The set of projective lines in $\mathbb{P}^3$ can be naturally identified with the \emph{Pl\"ucker quadric} $Gr(2,4)$ in $\mathbb{P}^5$: the line passing through points $x_1:y_1:z_1:w_1$ and $x_2:y_2:z_2:w_2$ is identified with the point
$x_1y_2-x_2y_1:x_1z_2-x_2z_1:x_1w_2-x_2w_1:y_1z_2-y_2z_1:y_1w_2-y_2w_1:z_1w_2-z_2w_1$.

Let $\Phi\subset\mathbb{R}^3$ be a surface covered by a family of real line segments and a family of real circular arcs simultaneously. The complex lines and complex circles containing the members of these families 
are called {\it generating lines} and {\it generating circles}, respectively.
Hereafter assume that $\Phi\subset\mathbb{R}^3$ is not a plane.

\begin{lemma} \label{parallel} The planes of the generating circles are parallel to each other.
\end{lemma}

To prove the lemma, we need several auxiliary claims. The first one is essentially known.



\begin{claim} \label{harris} \textup{(See \cite[Example~8.1]{Harris})}
Let $\gamma_1, \gamma_2, \gamma_3\subset \mathbb{P}^3$ be pairwise distinct irreducible algebraic curves.

\noindent\textup{(1)} The set $J(\gamma_1)\subset Gr(2,4)$ of all the lines passing through the curve $\gamma_1$ is an algebraic subset of $Gr(2,4)$.

\noindent\textup{(2)} The set ${J}(\gamma_1,\gamma_2)\subset Gr(2,4)$ of all the lines passing through each of the curves $\gamma_1, \gamma_2$ but not passing through their intersection $\gamma_1\cap \gamma_2$ is a piece of a $2$-dimensional algebraic surface
in $Gr(2,4)$.

\noindent\textup{(3)} The union $J (\gamma_1,\gamma_2,\gamma_3)\subset\mathbb{P}^3$ of all the lines passing through each of the curves $\gamma_1, \gamma_2, \gamma_3$ but not passing through their pairwise intersections is a piece of an algebraic surface
in $\mathbb{P}^3$.
\end{claim}

\begin{remark*}
Moreover, one can check 
that ${J}(\gamma_1,\gamma_2)$ and $J (\gamma_1,\gamma_2,\gamma_3)$ are quasi-projective varieties.
\end{remark*}

\begin{proof}
(1) The set $S$ of pairs $(P,\lambda)\in \mathbb{P}^3\times Gr(2,4)$ such that the point $P$ belongs to the line $\lambda$ is algebraic. The set $J(\gamma_1)$ is the image of the projection $S\cap (\gamma_1\times Gr(2,4))\to Gr(2,4)$ and by ~\textup{~\cite[Theorem 3.12]{Harris}} is also algebraic.


(2) Consider the polynomial map $(\gamma_1-\gamma_1\cap\gamma_2)\times (\gamma_2-\gamma_1\cap\gamma_2)\to Gr(2,4)$ taking a pair of points $(P,Q)$ to the line passing through $P$ and $Q$. Then ${J}(\gamma_1,\gamma_2)$ is the image of this map. The map $(\gamma_1-\gamma_1\cap\gamma_2)\times (\gamma_2-\gamma_1\cap\gamma_2)\to
{J}(\gamma_1,\gamma_2)$ is a finite covering by
a piece of a 2-dimensional irreducible surface. Thus
${J}(\gamma_1,\gamma_2)$ is a piece of a 2-dimensional irreducible algebraic surface ${\bar{J}}(\gamma_1,\gamma_2)\subset Gr(2,4)$.
(3) By (2) the set $\bigcap
_{i\ne j} {J} (\gamma_i,\gamma_j)$
is a (possibly empty) piece of the algebraic set $\bigcap
_{i\ne j} {\bar{J}} (\gamma_i,\gamma_j)$. Since for each $i\ne j$  the surface ${\bar{J}} (\gamma_i,\gamma_j)$ is 2-dimensional and irreducible
it follows that $\dim \bigcap
_{i\ne j} {\bar{J}} (\gamma_i,\gamma_j)\leq 2$.
If $\dim \bigcap
_{i\ne j} {\bar{J}} (\gamma_i,\gamma_j)=2$
then ${\bar{J}} (\gamma_1,\gamma_2)={\bar{J}} (\gamma_2,\gamma_3)={\bar{J}} (\gamma_3,\gamma_1)$
and by~\textup{~\cite[Theorem 1, $n=3$]{Kaminski}} it follows that $J (\gamma_1, \gamma_2, \gamma_3)$ is a piece of a plane.
If $\dim\bigcap
_{i\ne j} {\bar{J}} (\gamma_i,\gamma_j)\leq 1$ then
$J (\gamma_1, \gamma_2, \gamma_3)$ is a piece of an algebraic surface as the image of the projection $S\cap (\mathbb{P}^3\times \bigcap_{i\ne j} J(\gamma_i,\gamma_j)) \to \mathbb{P}^3$.
\end{proof}

\begin{figure}[htbp]
\centering
\begin{overpic}[width=0.33\textwidth]{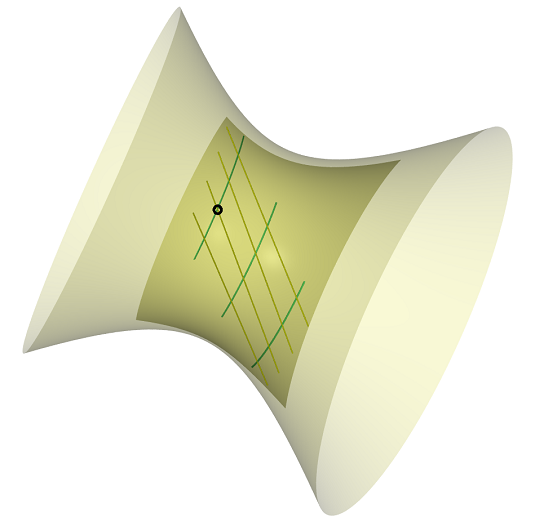}
\put(12, 36){$\bar{\Phi}$}
\put(63, 60) {$\Phi$}
\put(30, 49) {$P$}
\put(45, 66) {$\gamma_1$}
\put(51, 59) {$\gamma_2$}
\put(55, 47) {$\gamma_3$}
\end{overpic}
 \caption{To the proof of Claim~\ref{algebraic}.}

\label{Cl2_3}
\end{figure}
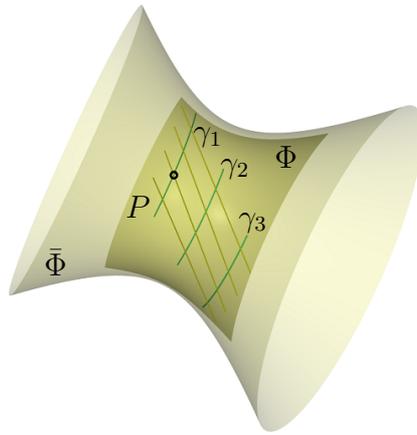




\begin{claim} \label{algebraic} \textup{(Cf.~\cite[Theorems 1 and 2]{Schicho-01})} The surface $\Phi\subset\mathbb{R}^3$ is contained in an irreducible ruled algebraic surface $\bar \Phi\subset \mathbb{P}^3$. The family of generating lines is a piece of an irreducible algebraic curve in $Gr(2,4)$.
\end{claim}

\begin{proof}
Take a point $P\in \Phi$; see Figure~\ref{Cl2_3}. Draw a circular arc $\gamma_1\subset \Phi$ through the point $P$.
Draw the line segments in $\Phi$ from our continuous family through each point of $\gamma_1$. 
Since the drawn segments are transversal to $\gamma_1$, the circular arcs contained in $\Phi$ form a continuous family, and $\Phi$ is smooth it follows that there are arcs $\gamma_2,\gamma_3\subset \Phi$ (sufficiently close to $\gamma_1$) which intersect each of the drawn segments.
Let $\bar \gamma_1,\bar \gamma_2,\bar \gamma_3$ be the complex circles in $\mathbb{P}^3$ containing the arcs $\gamma_1,\gamma_2,\gamma_3$.
Let $\Phi'\subset\Phi$ be the union of those drawn segments,
whose lines do not pass through the intersections $\bar \gamma_1\cap\bar \gamma_2$, $\bar \gamma_2\cap\bar \gamma_3$, $\bar \gamma_3\cap\bar \gamma_1$.
By construction $\Phi'\subset J(\bar \gamma_1,\bar \gamma_2,\bar \gamma_3)$.
Thus by Claim~\ref{harris}(3) the collar $\Phi'$ is a piece of an algebraic surface. 
Take $\bar \Phi$ to be an irreducible component of the surface 
containing 
a closed 2-dimensional subset of the initial surface $\Phi$ including the point $P$. 
(If there are no such components, e.g., $\Phi'=\emptyset$, then the drawn segments sufficiently close to the point $P$ form a quadratic cone with vertex at one of the intersection points of the circles $\bar \gamma_1,\bar \gamma_2,\bar \gamma_3$; in this case set $\bar\Phi$ to be this cone.) The algebraic surface $\bar \Phi\subset\mathbb{P}^3$ does not depend on the point $P$ because the smooth surface $\Phi\subset\mathbb{R}^3$ cannot jump from one irreducible algebraic surface to another.

By Claim~\ref{harris}(2) the lines containing the drawn segments form a piece of the algebraic set $\bigcap_{i\ne j} {\bar J} (\gamma_i,\gamma_j)$. Since $\Phi$ is not a plane it follows that the latter set is an algebraic curve \cite[Theorem 1, $n=3$]{Kaminski}. Take $\alpha\subset Gr(2,4)$ to be an irreducible component of this curve containing the lines sufficiently close to the point $P$. Clearly, the union of the lines of the whole curve $\alpha$ covers $\bar \Phi$, i.e., $\bar\Phi$ is ruled. It remains to show that the curve $\alpha$ does not depend on the choice of the point $P$. Indeed, assume that the generating lines through a neighborhood of another point $P'$ form a curve $\alpha'\subset Gr(2,4)$ distinct from $\alpha$. Then $\bar \Phi$ is doubly ruled and hence it is a quadric. Thus $\alpha\cap\alpha'=\emptyset$ and hence the generating lines cannot form a continuous family. This contradiction proves the claim.
\end{proof}

Hereafter any line belonging to the irreducible algebraic curve in $Gr(2,4)$ containing the generating lines is also called a \emph{generating line}. No confusion will arise from this.

\begin{claim} \label{linecross} \textup{(Cf. \cite[Lemma~1.3]{Polo-Blanko-Put-Top-09})} If $\gamma\subset\bar\Phi$ is an irreducible algebraic curve distinct from a generating line then each generating line intersects $\gamma$.
\end{claim}

\begin{proof}
Since $\bar\Phi\subset\mathbb{P}^3$ is ruled it follows that
there is a generating line through each point of $\gamma$.
Thus infinitely many generating lines belong to $J(\gamma)$. By Claim~\ref{harris}(1) the set $J(\gamma)$ is algebraic, hence the whole irreducible algebraic curve in $Gr(2,4)$ formed by generating lines is contained in $J(\gamma)$. \end{proof}

\begin{claim} \label{doesnotcontain} The surface $\bar \Phi\subset\mathbb{P}^3$ does not contain the absolute conic.
\end{claim}

\begin{proof}
Assume that $\bar\Phi$ contains the absolute conic $\gamma$. Then by Claim~\ref{linecross} all the generating lines intersect $\gamma$. Since $\gamma$ has no real points it follows that there are no real generating lines (except infinitely distant). This contradiction proves the claim.
\end{proof}

\begin{proof}[Proof of Lemma~\ref{parallel}]
By Claims~\ref{algebraic} and~\ref{doesnotcontain} the intersection of the surface $\bar \Phi\subset\mathbb{P}^3$ with the absolute conic is a finite set $I$. The plane of each generating circle intersects the infinitely distant plane by a line joining two points of the set $I$. Since the set $I$ is finite and the family of generating circles is continuous it follows that all these lines coincide, that is, all the planes of the circles are parallel.
\end{proof}

\begin{lemma} \label{circleonly} There are
infinitely many generating circles $\gamma$ such that the
projective plane $\Pi\subset\mathbb{P}^3$ of $\gamma$ intersects the
surface $\bar\Phi\subset\mathbb{P}^3$ only at the points of~$\gamma$.
\end{lemma}

To prove the lemma, we need the following auxiliary claim.

\begin{claim} \label{nogenerating} The projective planes $\Pi\subset\mathbb{P}^3$ of infinitely many generating circles $\gamma$ do not contain generating lines.
\end{claim}

\begin{proof}
Assume that the projective planes of only finitely many generating circles do not contain generating lines.
Thus the projective planes $\Pi\subset\mathbb{P}^3$ of infinitely many generating circles $\gamma$ contain generating lines $\lambda_\gamma$.
By Lemma~\ref{parallel} all the projective planes $\Pi$ intersect
the absolute conic by the same $2$-point set $I=\{P,Q\}$. It suffices to consider the following $3$ cases.

Case 1: \emph{For some $\gamma$ we have $\lambda_\gamma\cap I=\emptyset$}. Take a generating circle $\gamma'\not\subset\Pi$. Then $\Pi\cap\gamma'=I$ by Lemma~\ref{parallel}. Then by Claim~\ref{linecross} we have $\emptyset\ne \lambda_\gamma\cap\gamma'\subset \Pi\cap\gamma'-I=\emptyset$, a contradiction.

Case 2: \emph{For infinitely many $\gamma$ the intersection $\lambda_\gamma\cap I$ consists of a single point}. All the lines $\lambda_\gamma$ with this property are pairwise distinct because by Lemma~\ref{parallel} they are contained in the planes through $I$. We get infinitely many generating lines intersecting $I$. Thus by Claim~\ref{algebraic} \emph{each} generating line must intersect $I$. Then the generating lines through each point of a generating circle $\gamma$ must intersect $I$, a contradiction.

\begin{figure}[htbp]
\centering
\begin{overpic}[width=0.72\textwidth]{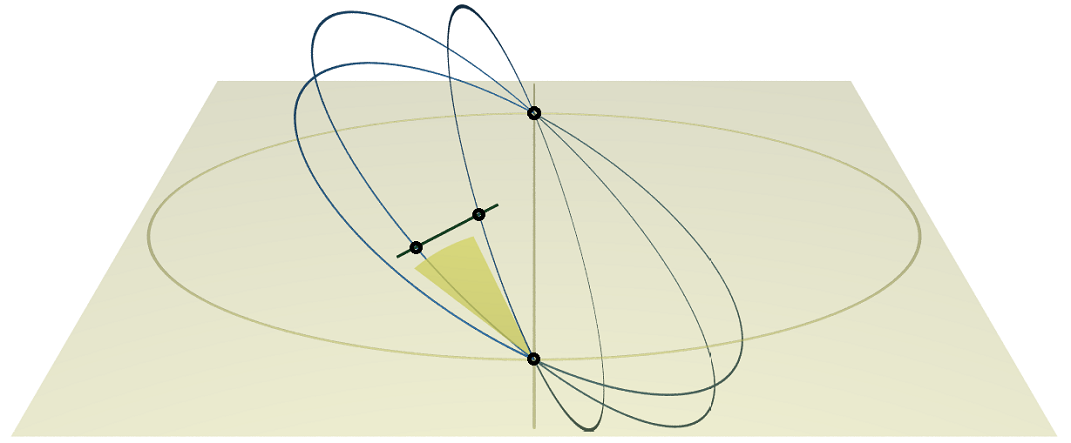}
\put(50, 9) {$P$}
\put(46.9, 28) {$Q$}
\put(37.5, 20.5) {$P_{t}^{2}$}
\put(44, 23.5) {$P_{t}^{1}$}
\put(43, 14) {$\Omega$}
\put(32.6, 14.8) {$\lambda_t$}

\put(26.9, 22) {$\gamma_3$}
\put(33, 26) {$\gamma_2$}
\put(39.6, 27.6) {$\gamma_1$}

\put(80.2, 5.3) {$w=0$}

\put(45.8, 3) {$\lambda_\gamma$}

\end{overpic}
\caption{To the proof of Claim~\ref{nogenerating} case (3).}
\label{Cl2_7}
\end{figure}
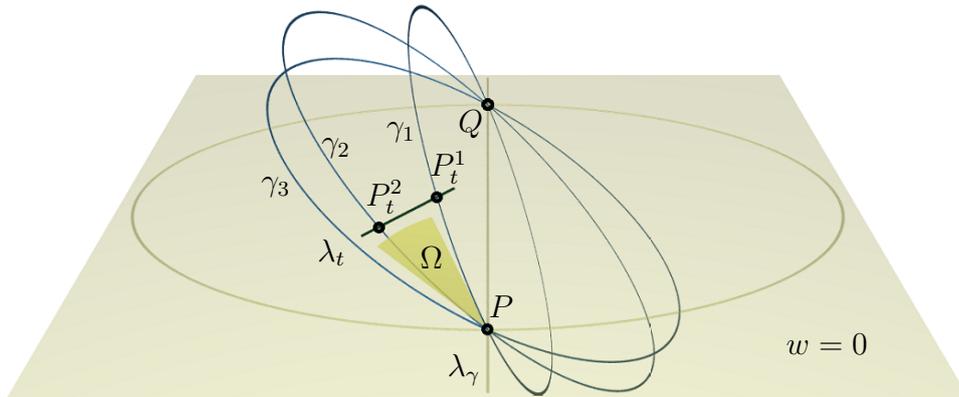
Case 3: \emph{For some $\gamma$ we have $\lambda_\gamma\cap I=I$}; see Figure~\ref{Cl2_7}. Then $\lambda_\gamma$ is the infinitely distant line 
of the projective plane $\Pi$.
Since the generating lines form an algebraic curve in $Gr(2,4)$ it follows that there is a sequence of generating lines
$\lambda_t\ne \lambda_\gamma$ converging to $\lambda_\gamma$.
Since there are only finitely many generating lines crossing $I$, we may assume that $\lambda_t\cap I=\emptyset$. Take $3$ pairwise noncoplanar generating circles $\gamma_1$, $\gamma_2$, and $\gamma_3$.
By Claim~\ref{linecross} for each $i=1,2,3$ the line $\lambda_t$ intersects the circle $\gamma_i$ at some point $P_t^i$. Each of the $3$ points $P_t^i$ converges to one of the $2$ points of the set $I$. By the pigeonhole principle
we may assume that, say, $P_t^1,P_t^2$ converge to $P$.
Then the plane $PP_t^1P^2_t$ converges to the projective plane $\Omega$ containing projective tangent lines to $\gamma_1$ and $\gamma_2$ at
the point $P$ (the tangent lines are distinct because $\gamma_1$ and $\gamma_2$ are not coplanar).
The projective plane $\Omega$ has a unique common point with $\gamma_1$ while $\lambda_\gamma\subset\Omega$ intersects $\gamma_1$ by the $2$-point set $I$. This contradiction proves the claim.
\end{proof}

\begin{proof}[Proof of Lemma~\ref{circleonly}] 
By Claim~\ref{nogenerating} there are infinitely many generating circles $\gamma$ such that $\bar\Phi\cap\Pi$ does not contain generating lines. Then $\bar\Phi\cap\Pi=\gamma$ because by Claim~\ref{linecross} the generating line through each point of $\bar\Phi\cap\Pi$ crosses $\gamma$.
\end{proof}

\begin{proof}[Proof of Theorem~\ref{ruled}]
By Lemma~\ref{circleonly} we have $\bar\Phi\cap\Pi=\gamma$ for infinitely many generating circles $\gamma$. So there is a generating circle $\gamma$ with this property which is not a singular curve of the surface $\bar\Phi$, because $\bar\Phi$ contains only finitely many singular curves. By Claim~\ref{nogenerating} the plane $\Pi$ does not touch the surface $\bar\Phi$ along the curve $\gamma$. Thus the circle $\gamma$ has multiplicity $1$ in the curve $\bar\Phi\cap\Pi$. By the Bezout theorem \cite[Theorem 18.3]{Harris} the degree of the surface $\bar\Phi\subset\mathbb{P}^3$ equals to the degree of its planar section (with multiplicity), and thus equals to $2$. Since $\bar\Phi$ contains both real lines and real circles, it is either a one-sheeted hyperboloid, or a quadratic cone, or an elliptic cylinder. Theorem~\ref{ruled} is proved.
\end{proof}


\section{Variations}\label{sec:variations}

In this subsection we prove some related results. Throughout this section we work over the field of real numbers except otherwise explicitly indicated. First let us give a simple proof of the following theorem.

\begin{theorem}[Takeuchi, 1987, \cite{Takeuchi}] \label{takeuchi} Let $\Phi\subset\mathbb{R}^3$ be a smooth closed surface homeomorphic to either a sphere or a torus. If through each point of the surface one can draw at least $7$ distinct circles fully contained in the surface (and continuously depending on the point) then the surface is a round sphere.
\end{theorem}

The proof just like the original one is topological. We simplify the original proof drastically using homology modulo $2$ instead of integral homology.

\begin{claim} \label{topological} From any $n$ smooth closed curves in the surface $\Phi$ intersecting pairwise in finitely many points
one can choose greater or equal to $n/3$ curves intersecting pairwise in an {even} number of points (counted with multiplicities).
\end{claim}

\begin{proof} 
If $\Phi$ is topologically a sphere then any two curves in $\Phi$ intersect in an {even} number of points (with multiplicities), and the claim follows.
Assume that $\Phi$ is topologically a torus, and $\xi$ and $\eta$ are its ``meridian'' and ``parallel'' curves.
Then the 1st homology group of $\Phi$ with coefficients modulo $2$ consists of $4$ homology classes: $0$, $[\xi]$, $[\eta]$, $[\xi]+[\eta]$. By the pigeonhole principle it follows that one of the $3$ sets $\{0,[\xi]\}$, $\{0,[\eta]\}$, $\{0,[\xi]+[\eta]\}$ contains the homology classes of greater or equal to $n/3$ of the given $n$ curves. Since $[\xi]\cap[\xi]=[\eta]\cap[\eta]=
([\xi]+[\eta])\cap ([\xi]+[\eta])=0$ it follows that the curves from the same set intersect each other in an even number of points.
\end{proof}

\begin{figure}[htbp]
\centering
\begin{overpic}[width=0.33\textwidth]{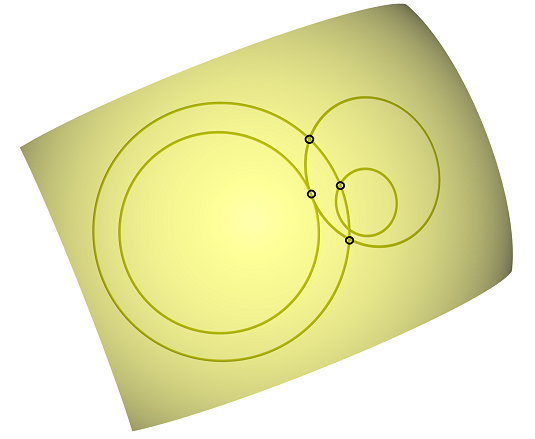}
\put(25, 6) {$\Phi$}
\put(51, 41) {$P$}
\put(65, 43) {$P'$}
\put(48, 28) {$\alpha_{0}$}
\put(60, 21) {$\alpha_{\epsilon}$}
\put(74, 43) {$\beta_{\epsilon}$}
\put(81, 40) {$\beta_{0}$}
\end{overpic}
\caption{To the proof of Claim~\ref{generic}.}
\label{Cl3_3}
\end{figure}
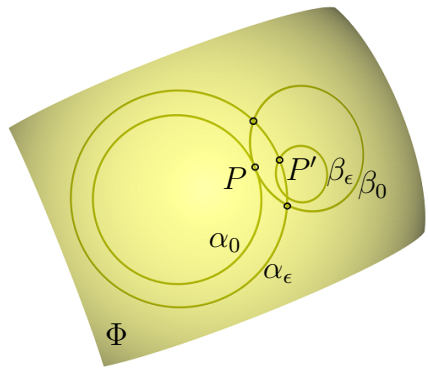

\begin{claim}\label{generic} Two circles passing through a generic point $P\in \Phi$ are transversal.
\end{claim}

\begin{proof} Assume that the circles $\alpha_0$ and $\beta_0$ passing trough the point $P$ touch each other; see Figure~\ref{Cl3_3}. Move the point $P$ slightly along a geodesic in the normal direction to the circles. Draw the circles $\alpha_{\epsilon}$ and $\beta_{\epsilon}$ through the resulting point $P'$ of $\Phi$.
Then one of the pairs $\alpha_{\epsilon}, \beta_{0}$ or $\alpha_{0}, \beta_{\epsilon}$ has $2$ distinct intersection points near $P$. Hence there are points arbitrarily close to $P$ such that the circles through them are transversal.
\end{proof}

\begin{proof}[Proof of Theorem~\ref{takeuchi}] 
The parity of the number of intersection points (with multiplicities) does not depend on the choice of particular curves from the families. Thus by Claim~\ref{topological} there are $3$ families $\alpha_t$, $\beta_t$, $\gamma_t$ of circles in $\Phi$ such that the circles from distinct families intersect in an even number of points (with multiplicities). Assume without loss of generality that the circles $\alpha_0$, $\beta_0$, $\gamma_0$ pass through a generic point $P\in \Phi$.
Then by Claim~\ref{generic} the circles $\alpha_0$, $\beta_0$, $\gamma_0$ intersect transversally.

Then for sufficiently small $t$ any two of the circles $\alpha_0$, $\beta_0$, $\gamma_t$ have exactly two common points. Circles with $2$ common points are cospherical. Thus $\alpha_0$ and $\beta_0$ belong to one sphere. Each circle $\gamma_t\not\ni P$ belongs to the same sphere because it has at least $3$ common points with the sphere ($2$ pairs of intersection points with $\alpha_0$ and $\beta_0$, minus at most $1$ common point of the pairs).
Since $\Phi$ is a smooth surface, it cannot jump from one sphere to another. Thus $\Phi$ is a sphere itself; cf.~\cite[Theorem~1]{Takeuchi-85}.
\end{proof}

Now let us prove a new result generalizing an old one of Darboux.

\begin{theorem}\label{3families}
Let $\Phi\subset\mathbb{R}^3$ be a smooth closed surface homeomorphic to either a sphere or a torus.
If through each point of the surface one can draw at least $4$ distinct circles fully contained in the surface (and continuously depending on the point) then the surface is a cyclide.
\end{theorem}

\begin{proof}[Proof of Theorem~\ref{3families}] Let $\alpha_t$, $\beta_t$, $\gamma_t$, $\delta_t$ be the $4$ families of circles in the surface $\Phi$. Assume without loss of generality that the circles $\alpha_0$, $\beta_0$, $\gamma_0$, $\delta_0$ pass through a generic point $P$.
Since $\Phi$ is topologically either a sphere or a torus by~Claim~\ref{topological} it follows that we can choose $2$ families, say, $\alpha_t$ and $\beta_t$, such that for each $s,t$ the circles $\alpha_t$ and $\beta_s$ have an even number of intersection points (with multiplicities). By Claim~\ref{generic} the circles $\alpha_0$ and $\beta_0$ are transversal. Thus for $s$, $t$ sufficiently close to zero the circles $\alpha_t$ and $\beta_s$ are cospherical. 
Now the theorem follows from the following classical result.
\end{proof}

\begin{theorem}\label{cospherical} \textup{(See~\cite[Theorem 20 in p.~296]{Coolidge})}
If through each point of a smooth surface in $\mathbb{R}^3$ one can draw $2$ cospherical circular arcs fully contained in the surface (and continuously depending on the point) then the surface is a cyclide.
\end{theorem}

\begin{proof} [Sketch of the proof]
The surface is covered by two families of circular arcs $\alpha_t$ and $\beta_t$.
Fix a sphere $\Sigma$ orthogonal to both $\alpha_1$ and $\alpha_2$.
Take the spheres $\Sigma_{1t}$ and $\Sigma_{2t}$ containing $\alpha_1\cup\beta_t$ and $\alpha_2\cup\beta_t$, respectively. Since $\alpha_1 \perp \Sigma$ it follows that $\Sigma_{1t} \perp \Sigma$. Analogously,
$\Sigma_{2t} \perp \Sigma$. If $\Sigma_{1t}=\Sigma_{2t}$ for each $t$ then $\Phi$ is a sphere through the circles $\alpha_1$ and $\alpha_2$. Otherwise $\beta_t\subset\Sigma_{1t}\cap\Sigma_{2t} \perp \Sigma$ for each $t$. 
Analogously $\alpha_t\perp \Sigma$ for each $t$. Assume without loss of generality that $\Sigma$ is the unit sphere. Perform the {\it Darboux transformation} $(x,y,z)\mapsto 2(x,y,z)/(x^2+y^2+z^2+1)$ taking all circles orthogonal to $\Sigma$ to line segments. The transformation takes our surface to a doubly ruled surface. Thus our surface is the preimage of a quadric under the Darboux transformation, i.e., a cyclide.
\end{proof}

Now let us give examples showing nontriviality of our results.
We begin with examples over the field of complex numbers.

\begin{example} \label{ex1} The irreducible complex cyclide $\left(x^2+y^2+z^2\right)^2+(x+iy)^2-z^2=0$, which can be parametrized in $\mathbb{P}^3$ as $t^2-1:i(t^2-1-2st):s(t^2+1):s(t^2-1)+4t$,
is covered by a family of complex lines $t=\mathrm{const}$
and a family of complex circles  $s=\mathrm{const}$ simultaneously.
\end{example}




\begin{example}\label{ex3} A general position degree $3$ complex cyclide is covered by a family of complex circles and contains $27$ complex lines; however, the surface contains no families of complex lines.
\end{example}

\begin{proof} Any cyclide is covered by at least one family of complex circles \cite[Chapter~VII]{Coolidge}. A general position degree $3$ cyclide is nonsingular and hence contains exactly $27$ complex lines.
\end{proof}

\begin{example} \label{ex2} A general position ruled complex cubic surface 
is covered by a family of complex lines and contains $15$ complex circles; however, the surface contains no families of complex circles.
\end{example}

\begin{proof}[Proof of Example~\ref{ex2}] Consider the intersection $I$ of the ruled cubic surface with the absolute conic. In general position it consists of six distinct points. Let $P$, $Q$ be two of the intersection points. Let $\lambda_1$ be the line passing through $P$, $Q$, and let $R$ be the third common point of $\lambda_1$ and the surface. In general position $R\ne P,Q$.
Consider the ruling $\lambda_2$ passing through $R$. Take the plane $\Pi$ containing the lines $\lambda_1$ and $\lambda_2$. The intersection of $\Pi$ and the surface consists of the ruling $\lambda_2$ and a curve $\gamma$ of degree $2$. The curve $\gamma$ is irreducible once the plane $\Pi$ contains neither the singular line nor the isolated line of the surface,
i.e., the points $P$, $Q$, $R$ do not belong to these lines (this follows from the classification of ruled cubic surfaces \cite[Section 2]{Polo-Blanko-Put-Top-09}). Thus in general position $\gamma$ is a conic through $P$ and $Q$, i.e., a complex circle. There are $15$ ways to choose two distinct points $P,Q\in I$ leading to $15$ complex circles on the surface.
\end{proof}

Finally, let us proceed to examples over the field of real numbers. Their obvious proofs are omitted.

\begin{example} \label{ex5} (See Figure~\ref{fig-examples} to the left.)
The surface $(x^2-z^2)(3z-2)+(y-z)(3yz-2y-4z+2)=0$ is covered by a family of circles in the planes $z=\mathrm{const}$ and contains $4$ lines: $l_1(t)=(t,t,t)$, $l_2(t)=(-t,t,t)$, $l_3(t)=({t},1-{t},2t)$, $l_4(t)=(-{t},1-{t},2t)$; however, the surface contains no families of lines.
\end{example}

\begin{example} \label{ex4} (See \cite[Section~1]{Pottmann} and Figure~\ref{fig-examples} in the middle.) The surface
 $(x^2+y^2+z^2+3)^2 - 4y^2z^2 - 16x^2 - 12y^2=0$ obtained by translation of a circle along another one is covered by $2$ families of circles in the planes $y=\mathrm{const}$ and $z=\mathrm{const}$ but it is not a cyclide.
\end{example}


It is not known so far if there are surfaces covered by $3$ families of circles besides cyclides. Let us give a related counterexample from \emph{isotropic geometry} \cite{Pottmann-2009}. An \emph{isotropic circle} is either a parabola with the axis parallel to~$Oz$ or an ellipse whose projection into the plane $Oxy$ is a circle. An \emph{isotropic cyclide} is given by equation~\eqref{D}, in which both instances of $(x^2+y^2+z^2)$ are replaced by $(x^2+y^2)$ \cite[Section~5]{Pottmann}.

\begin{example} \label{ex6} (See Figure~\ref{fig-examples} to the right.) The surface $z=xy(y-x)$ is covered by $3$ families of isotropic circles in the planes $x=\mathrm{const}$, $y=\mathrm{const}$, $y-x=\mathrm{const}$ but it is not an isotropic cyclide. 
\end{example}

\begin{figure}[htbp]
\includegraphics[width=0.36\textwidth]{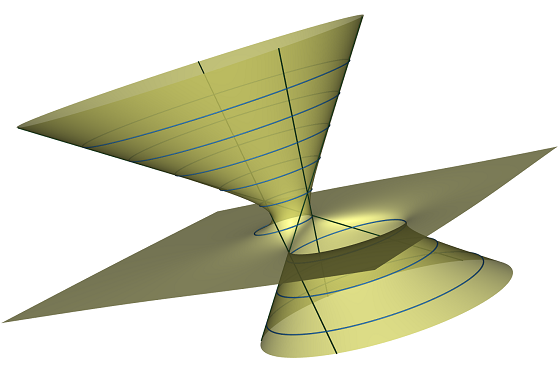}\qquad
\includegraphics[width=0.24\textwidth]{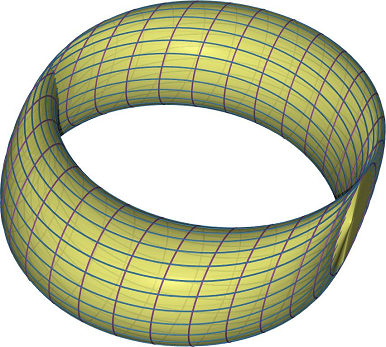}
\qquad
\includegraphics[width=0.22\textwidth]{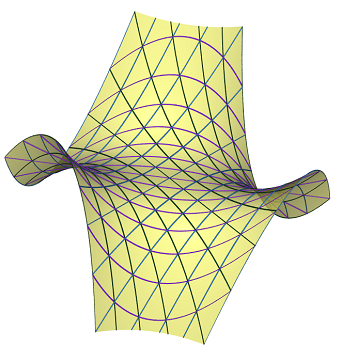}\qquad
\caption{
Left: a surface covered by $1$ family of circles and containing $4$ lines (Example~\ref{ex5}). Middle: a surface covered by $2$ families of circles (Example~\ref{ex4}). Right: a surface covered by $3$ families of isotropic circles (Example~\ref{ex6}).}
\label{fig-examples}
\end{figure}

\section{Acknowledgements}

The authors are grateful to S.~Gusein-Zade, R.~Krasauskas, H.~Pottmann,  I.~Sabitov for useful discussions and to L.~Shi for the first figure.

\end{document}